\documentclass[pdflatex,sn-mathphys-num]{sn-jnl}

\usepackage{graphicx}
\usepackage{amsmath,amssymb,amsfonts}
\usepackage{amsthm}
\usepackage{booktabs}
\usepackage{manyfoot}
\usepackage{placeins}
\let\orcidlogo\relax
\usepackage{orcidlink}

\makeatletter
\long\def\@tablecaption#1#2{%
  \setbox\tabcapbox\vbox{\tablecaptionfont\centering
    {\bfseries #1}\hskip2mm #2\vphantom{y}\par}%
  \box\tabcapbox
}
\makeatother
\DeclareNewFootnote{A}[gobble]
\newcommand{\Li}{\operatorname{li}}
\numberwithin{equation}{section}
\theoremstyle{thmstyleone}
\newtheorem{theorem}{Theorem}[section]
\newtheorem{lemma}[theorem]{Lemma}
\newtheorem{proposition}[theorem]{Proposition}
\newtheorem{corollary}[theorem]{Corollary}
\hypersetup{pdftitle={Inequalities Involving Higher-Degree Polynomial Functions in pi(x)},
  pdfauthor={Subham De}}

\begin{document}

\title[Inequalities for polynomials in \(\pi(x)\)]{Inequalities Involving Higher-Degree Polynomial Functions in \(\pi(x)\)}
\author*[1]{\fnm{Subham} \sur{De}\textsuperscript{\orcidlink{0009-0001-3265-4354}}}\email{subham581994@gmail.com}
\affil*[1]{\orgname{Indian Institute of Technology Delhi},
  \orgaddress{\country{IN}}
  (\href{https://www.sites.google.com/view/subhamde}{Webpage})}

\artnote{Preprint versions are available on
\href{https://doi.org/10.48550/arXiv.2407.18983}{arXiv}
and \href{https://www.preprints.org/manuscript/202407.1276}{Preprints.org}.}
\miscnote{\textbf{Funding}\quad This research received no specific grant
from any funding agency in the public, commercial, or not-for-profit sectors.}

\abstract{
Earlier research reflects that the two sides of Ramanujan's inequality for the prime-counting
function
\begin{align*}
    \pi(x)^2<\frac{ex}{\log x}\pi(x/e)
\end{align*}
happens to agree through four asymptotic orders. In this paper, we try to examine how
this cancellation arises in higher-degree expressions and how
shifts affect the size of the resulting gaps. In signed sums of
equal powers of shifted prime counts, the first nonvanishing
moment of the logarithmic shifts determines the leading term.
A cubic comparison at five dilations exhibits four successive
cancellations, leaving a gap of order \(x^3/(\log x)^7\). A finite-difference
argument proves its positivity for every real \(x\geq e^{52}\)
when the logarithmic shifts are spaced by \(2\).
When the shifts and nonnegative weights are fixed, the weighted
Ramanujan gap retains order \(x^2/(\log x)^6\). Its leading coefficient is
\(c^2(1+\sigma^2)\), where \(c\) is the total rescaled weight
and \(\sigma^2\) is the weighted variance of the logarithmic
shifts. We also give the explicit range \(x\geq 85e\)
for the logarithmically weighted
five-shift inequality and determine the eventual sign of a
polynomial family in every fixed integer degree \(m\geq 3\).
Exact counts and interval enclosures illustrate the expansions
and the effect of variance at fixed rescaled weight and mean.
}

\keywords{Prime-Counting Function, Ramanujan's Inequality, Prime Number Theorem, Shifted Prime Counts, Weighted Sums}

\pacs[AMS subject classifications]{11A41, 11N05, 11N37, 11N56, 11M06, 11M26}

\maketitle

\section{Introduction}\label{sec:intro}
Ramanujan's letters to G.~H. Hardy in 1913 included formulas and
assertions about the distribution of primes \cite[pp.~xxiii--xxx,
349--353]{ramanujan1962}. Hardy later examined a related
prime-counting formula \cite{hardy1937,hardy1967}. Among the
comparisons recorded in Ramanujan's notebooks is the inequality
\begin{equation}\label{eq:ramanujan}
  \pi(x)^2 < \frac{ex}{\log x}\pi(x/e)
\end{equation}
for all sufficiently large \(x\), where \(\pi(x)\) counts the
primes at most \(x\) \cite[p.~112]{berndt1994}. The common
leading term is \(x^2/(\log x)^2\). The expansion in
Section~\ref{sec:gap} places the first difference on the
smaller scale \(x^2/(\log x)^6\).

The size of the gap makes the lower endpoint of the valid range
a separate problem. Hassani gave a range assuming the Riemann
hypothesis \cite[Theorem~1.2]{hassani2012}. Under the same
hypothesis, Dudek and Platt identified \(38\,358\,837\,682\)
as the largest integer counterexample
\cite[Theorem~1.3]{dudekplatt2015}. Zero-free regions for the
Riemann zeta function, such as the region of Mossinghoff and
Trudgian \cite{mossinghofftrudgian2015}, provide input for
explicit prime-counting estimates. Axler gives an unconditional
range beginning at \(\exp(3158.442)\)
\cite[Theorem~1.1]{axler2022}. The bound \(x\geq e^{490}\)
obtained by the Mertens-function argument in
\cite[Theorem~6.4]{de2026} supplies a further example of an
explicit range to which the transfer theorem of
Section~\ref{sec:sums} applies.

Hassani studied fixed dilations in \eqref{eq:ramanujan}
\cite{hassani2012} and iterated the inequality to obtain comparisons
involving \(\pi(x)^{2^n}\) \cite{hassani2013}. For cubes, he proved
\cite[Theorem~5]{hassani2021}
\begin{equation}\label{eq:hassani-cubic}
  \left(\frac{\sqrt e\,x}{\log x}\right)^2\pi(x/e)
  <\pi(x)^3<\frac{e^2 x}{\log x}\pi(x/e)^2
\end{equation}
for all sufficiently large \(x\). In the cubic and quartic
expressions of Section~\ref{sec:polynomials}, the leading
coefficients already determine the eventual signs. The comparisons
become more delicate when the shifted terms are chosen to cancel.
Section~\ref{sec:framework} gives a general rule for the first two
coefficients. For sums of equal powers with a common leading scale,
the vanishing moments of the logarithmic shifts determine how far
cancellation persists. This gives a cubic comparison whose first
four coefficients vanish. Corollary~\ref{cor:effective-cubic}
proves positivity for all real \(x\geq e^{52}\) at logarithmic
spacing \(2\).
The proof bounds a fourth finite difference of a smooth
logarithmic-integral model and then controls the error in the
prime counts. The same argument gives positivity at spacing \(1\)
for \(x\geq e^{2704}\).

Theorem~\ref{thm:weighted-coefficient} evaluates the weighted
gap at order \(x^2/(\log x)^6\). For fixed nonnegative weights,
the individual Ramanujan gaps contribute the square of the
total weight rescaled by the inverse
dilations. Variation among the normalized counts contributes this
same square multiplied by the weighted variance of the logarithmic
shifts. Cauchy--Schwarz also gives a transfer theorem for the sign
of this gap from a known range for Ramanujan's inequality.

For the logarithmically weighted sum with five shifts,
Corollary~\ref{cor:logweighted} establishes the inequality for
every real \(x\geq 85e\) using the classical estimates of
Rosser and Schoenfeld \cite[Corollary~1]{rosserschoenfeld1962}.
The numerical section compares the expansions with exact counts
and with interval enclosures at large arguments. Its final
comparison holds the rescaled total weight and the mean of
the logarithmic shifts fixed to isolate the variance term.

\section{Estimates For Shifted Prime Counts}\label{sec:estimates}
Throughout, \(L=\log x\) and \(\Li(y)\) denotes the logarithmic
integral with its usual principal-value convention. The Chebyshev
functions are
\[
  \theta(y)=\sum_{p\leq y}\log p,\qquad
  \psi(y)=\sum_{p^j\leq y}\log p,
\]
where \(p\) runs over primes and \(j\) runs over positive integers.
The identity \(\psi(y)=\sum_{j\geq 1}\theta(y^{1/j})\) follows
by grouping prime powers according to their exponent. The
bound \(\theta(t)=O(t)\) then gives
\(\psi(y)-\theta(y)=O(\sqrt y\log y)\), since the terms with
\(j\geq 2\) have arguments at most \(\sqrt y\) and vanish for
\(j>\log y/\log 2\). These functions are standard tools in the
analytic proof of the prime number theorem
\cite{ingham1932,jameson2003,titchmarsh1986}. Bounds based
directly on \(\theta\) are also discussed by De \cite{de2023}.

For \(y\geq 2\), partial summation with the weight \(1/\log t\) gives
\begin{equation}\label{eq:chebyshev-pi}
  \pi(y)=\frac{\theta(y)}{\log y}
            +\int_2^y\frac{\theta(t)}{t(\log t)^2}\,dt
         =\frac{\psi(y)}{\log y}
            +O\!\left(\frac{y}{(\log y)^2}\right).
\end{equation}
For the second equality, bound the integral using
\(\theta(t)=O(t)\) and split it at \(\sqrt y\). The
prime-power correction after division by \(\log y\) is
\(O(\sqrt y)\). Schoenfeld gives explicit estimates for the
Chebyshev functions, including bounds assuming the Riemann
hypothesis \cite{schoenfeld1976}. The expansions below use the
classical unconditional prime number theorem in the form used
for shifted counts in \cite{hassani2013,hassani2021}.

\begin{lemma}\label{lem:shift}
For fixed \(a>0\) and a nonnegative integer \(J\), as \(x\to\infty\),
\begin{equation}\label{eq:shift-full}
  \pi(x/a)=\frac{x}{a}\sum_{j=0}^{J}
  \frac{j!}{(L-\log a)^{j+1}}
  +O_{a,J}\!\left(\frac{x}{L^{J+2}}\right).
\end{equation}
In particular,
\begin{equation}\label{eq:shift-two}
  \pi(x/a)=\frac{x}{aL}
  \left(1+\frac{1+\log a}{L}+O_a(L^{-2})\right).
\end{equation}
\end{lemma}

\begin{proof}
The classical prime number theorem gives
\(\pi(y)=\Li(y)+O(y\exp(-c\sqrt{\log y}))\) for some \(c>0\).
Integration by parts in the integral defining \(\Li(y)\) gives the
terms \(j!y/(\log y)^{j+1}\) for \(0\leq j\leq J\). The remaining
integral is a constant multiple of
\(\int_2^y(\log t)^{-J-2}\,dt\), up to a constant from the lower
endpoint. Splitting this integral at \(\sqrt{y}\) bounds it by
\(O_J(y/(\log y)^{J+2})\). The prime number theorem error is of
smaller order. Substitution of \(y=x/a\) proves
\eqref{eq:shift-full}. Expanding the first two denominators in powers
of \(1/L\) gives \eqref{eq:shift-two}.
\end{proof}

All shifts and powers in the asymptotic arguments below are fixed
as \(x\) increases. For the numerical enclosures at large arguments
we use the explicit estimate from \cite[Lemma~6.2]{de2026},
\begin{equation}\label{eq:explicit-pi}
  |\pi(y)-\Li(y)|<E(y),
  \qquad E(y)=0.1y^{0.93}+21\,000y^{0.913},
  \qquad y\geq e^{489}.
\end{equation}
At a shifted argument \(y=x/a\), its range becomes \(x\geq ae^{489}\).

\section{Cubic And Quartic Comparisons}\label{sec:polynomials}
\subsection{The First Two Terms}
Define the cubic and quartic expressions
\begin{align}
  \mathcal H(x)
  &=
  \pi(x)^3-\frac{3ex}{L}\pi(x/e)^2
   +\frac{3e^2x}{L^2}\pi(x/e^2),\label{eq:cubic}\\
  \mathcal K(x)
  &=
  \pi(x)^4-\frac{4ex}{L}\pi(x/e)^3
   +\frac{6e^2x}{L^2}\pi(x/e^2)^2
   -\frac{4e^3x}{L^3}\pi(x/e^3).\label{eq:quartic}
\end{align}

\begin{theorem}\label{thm:polynomials}
As \(x\to\infty\),
\begin{align}
  \mathcal H(x)
  &=\frac{x^3}{L^3}
    \left[1-\frac{3}{e}+\frac{3-12/e}{L}
               +O(L^{-2})\right],\label{eq:cubic-leading}\\
  \mathcal K(x)
  &=\frac{x^4}{L^4}
    \left[1-\frac{4}{e^2}+\frac{4-24/e^2}{L}
               +O(L^{-2})\right].\label{eq:quartic-leading}
\end{align}
In particular, \(\mathcal H(x)<0\) and \(\mathcal K(x)>0\) for all
sufficiently large \(x\).
\end{theorem}

\begin{proof}
The first corrections relative to \(x/L\) in \(\pi(x)\)
and \(e\pi(x/e)\) are \(1/L\) and \(2/L\), respectively,
by \eqref{eq:shift-two}. Raising a factor
\(1+b/L+O(L^{-2})\) to a fixed power \(r\) changes
the coefficient \(b\) to \(rb\). Thus
\[
  \pi(x)^3=\frac{x^3}{L^3}
            \left(1+\frac{3}{L}+O(L^{-2})\right),\qquad
  \frac{3ex}{L}\pi(x/e)^2=\frac{3}{e}\frac{x^3}{L^3}
            \left(1+\frac{4}{L}+O(L^{-2})\right).
\]
The term involving \(\pi(x/e^2)\) has order \(x^2/L^3\).
After division by \(x^3/L^3\), its order is \(1/x\), which
is absorbed by the error in the bracket of
\eqref{eq:cubic-leading}. The quartic expansion uses
\[
  \pi(x)^4=\frac{x^4}{L^4}
           \left(1+\frac{4}{L}+O(L^{-2})\right),\qquad
  \frac{4ex}{L}\pi(x/e)^3=\frac{4}{e^2}\frac{x^4}{L^4}
           \left(1+\frac{6}{L}+O(L^{-2})\right).
\]
The other two terms of \eqref{eq:quartic} have orders
\(x^3/L^4\) and \(x^2/L^4\), so after normalization
they enter the remainder \(O(L^{-2})\). Since \(2<e<3\),
the cubic leading coefficient is negative and the quartic
leading coefficient is positive.
\end{proof}

\subsection{Comparison With The Ramanujan Gap}\label{sec:gap}
For a fixed \(a>0\), the two-term expansion
\eqref{eq:shift-two} gives
\begin{equation}\label{eq:dilation-gap}
  \pi(x)^2-\frac{a x}{L}\pi(x/a)
  =\frac{x^2}{L^3}(1-\log a)
   +O_a\!\left(\frac{x^2}{L^4}\right).
\end{equation}
Indeed, after division by \(x^2/L^2\), the first corrections
are \(2/L\) on the square and \((1+\log a)/L\) on the
shifted count. This shows why \(a=e\) is exceptional
in comparisons with fixed dilations
\cite{hassani2012} \cite[Theorem~3]{hassani2021}.

Write
\(\mathcal G(x)=\pi(x)^2-(ex/L)\pi(x/e)\) for the difference
in \eqref{eq:ramanujan}. To see the order of its cancellation,
set \(u=1/L\). Applying \eqref{eq:shift-full} to \(a=1,e\)
and expanding at \(u=0\) gives
\begin{align*}
  \left(\frac{L}{x}\pi(x)\right)^2
    &=1+2u+5u^2+16u^3+64u^4+O(u^5),\\
  \frac{eL}{x}\pi(x/e)
    &=1+2u+5u^2+16u^3+65u^4+O(u^5).
\end{align*}
In the second line the coefficient of \(u^j\), for \(0\leq j\leq 4\),
is \(\sum_{i=0}^j i!\binom{j}{i}\), obtained by expanding
\((1-u)^{-1}\sum_{i=0}^4 i!(u/(1-u))^i\).
Consequently,
\begin{equation}\label{eq:gap-order}
  \mathcal G(x)=-\frac{x^2}{L^6}+O\!\left(\frac{x^2}{L^7}\right),
\end{equation}
as also follows from the classical expansion used in \cite{hassani2013}.
An error of order \(x/L^2\) in either prime count can induce
an error of order \(x^2/L^3\) in this comparison. The nonzero
cubic coefficient in \eqref{eq:cubic-leading} and the cancellation
in \eqref{eq:gap-order} therefore require separate sign estimates.

There are also arguments at which the two signs differ. At
\(x=10\,193\), the integer parts of \(x/e\) and \(x/e^2\) are
\(3749\) and \(1379\), respectively. The three prime counts, in
order, are \(1252\), \(522\), and \(220\). Substitution with
interval bounds for \(e\) and \(\log x\) gives
\[
  425<\mathcal G(10\,193)<426,\qquad
  -490\,948\,000<\mathcal H(10\,193)<-490\,947\,000.
\]
The code accompanying the numerical section verifies these
enclosures.

\section{Finite Sums Of Shifted Counts}\label{sec:sums}
For an integer \(n\geq 1\), put
\begin{equation}\label{eq:sums}
  \begin{aligned}
    H_n&=\sum_{k=1}^n\frac{1}{k},&
    J_n&=\sum_{k=1}^n\frac{\log k}{k},\\
    S_n(x)&=\sum_{k=1}^n\pi(x/k),&
    W_n(x)&=\sum_{k=1}^n\frac{\pi(x/k)}{\log(x/k)}.
  \end{aligned}
\end{equation}
We use \(W_n\) for \(x>n\), so every denominator is positive.
The two quadratic comparisons, defined for \(x>en\), are
\begin{align}
  \mathcal L(x)&=S_n(x)^2-\frac{ex}{L}S_n(x/e),\label{eq:sum-gap}\\
  \mathcal F(x)&=W_n(x)^2-\frac{ex}{L}W_n(x/e).\label{eq:weighted-gap}
\end{align}
The dependence of \(\mathcal L\) and \(\mathcal F\) on \(n\) is
understood throughout.

\subsection{Asymptotic Signs And First Corrections}
\begin{proposition}\label{prop:sum-orders}
For fixed \(n\), as \(x\to\infty\),
\begin{align}
  \mathcal L(x)
    &=\frac{x^2}{L^2}\left[H_n(H_n-1)
      +\frac{2H_n(H_n+J_n)-(2H_n+J_n)}{L}
      +O_n(L^{-2})\right],\label{eq:sum-asymptotic}\\
  \mathcal F(x)
    &=\frac{x^2}{L^3}\left[-H_n
      +\frac{H_n^2-3H_n-2J_n}{L}
      +O_n(L^{-2})\right].\label{eq:weighted-asymptotic}
\end{align}
Thus \(\mathcal L(x)>0\) eventually when \(n>1\), and
\(\mathcal F(x)<0\) eventually for every \(n\geq 1\).
\end{proposition}

\begin{proof}
For \(j\in\{0,1,2\}\), apply \eqref{eq:shift-two} with
\(a=e^jk\) to each summand. For \(W_n\), multiply by
\((L-j-\log k)^{-1}\), whose relative correction to \(L^{-1}\)
is \((j+\log k)/L\). Summing over the fixed shifts gives
\begin{align}
  S_n(x/e^j)
    &=\frac{x}{e^jL}\left[H_n+\frac{(1+j)H_n+J_n}{L}
                                 +O_n(L^{-2})\right],
      \label{eq:shifted-S}\\
  W_n(x/e^j)
    &=\frac{x}{e^jL^2}\left[H_n+\frac{(1+2j)H_n+2J_n}{L}
                                   +O_n(L^{-2})\right].
      \label{eq:shifted-W}
\end{align}
In \(\mathcal L\), the square contributes \(H_n^2\) and
\(2H_n(H_n+J_n)/L\) after normalization. The term at \(x/e\)
contributes \(H_n\) and \((2H_n+J_n)/L\) on the same scale.
Their difference gives \eqref{eq:sum-asymptotic}. The leading
coefficient is positive for \(n>1\). For \(n=1\), both
displayed coefficients vanish and \(\mathcal L=\mathcal G\),
whose first nonzero term is given by \eqref{eq:gap-order}.

The factor \(ex/L\) in \(\mathcal F\) puts the shifted term
on the scale \(x^2/L^3\). The square contributes
\(H_n^2x^2/L^4\) at the next order. This gives
\eqref{eq:weighted-asymptotic}, with a negative leading
coefficient for every \(n\geq 1\).
\end{proof}

\subsection{An Explicit Inequality For Nonnegative Weights}
The following transfer uses \eqref{eq:ramanujan} pointwise at the
shifted arguments. It applies to weights that depend on \(x\), which
allows the result to be used with the logarithmic weights in
\eqref{eq:sums}.

\begin{theorem}\label{thm:transfer}
Suppose \(T>1\) and \eqref{eq:ramanujan} holds for every real
\(y\geq T\). Let \(a_1,\ldots,a_n>0\), and let
\(w_1,\ldots,w_n\) be nonnegative at the argument \(x\) under
consideration, with at least one positive. Then, for
\(x\geq T\max_k a_k\),
\begin{equation}\label{eq:transfer}
  \left(\sum_{k=1}^n w_k\pi(x/a_k)\right)^2
  <ex\left(\sum_{k=1}^n\frac{w_k}{a_k\log(x/a_k)}\right)
      \left(\sum_{k=1}^n w_k\pi(x/(ea_k))\right).
\end{equation}
\end{theorem}

\begin{proof}
Each \(y=x/a_k\) is at least \(T\). Taking square roots in
\eqref{eq:ramanujan} at that argument gives
\[
  \pi(x/a_k)
   <\sqrt{\frac{ex\,\pi(x/(ea_k))}{a_k\log(x/a_k)}}.
\]
Multiply by \(w_k\) and sum. At least one term has positive
weight, so the comparison is strict. By Cauchy--Schwarz,
\[
  \left(\sum_{k=1}^n
     \sqrt{\frac{w_k}{a_k\log(x/a_k)}}\,
     \sqrt{w_k\pi(x/(ea_k))}\right)^2
  \leq
  \left(\sum_{k=1}^n\frac{w_k}{a_k\log(x/a_k)}\right)
  \left(\sum_{k=1}^n w_k\pi(x/(ea_k))\right).
\]
This gives \eqref{eq:transfer} after inserting the factor \(ex\).
\end{proof}

\subsection{Explicit Ranges For Logarithmic Weights}
The two terms in \(\mathcal F\) have different asymptotic orders
by Proposition~\ref{prop:sum-orders}. This difference permits a
direct effective estimate using upper and lower bounds for the
prime counts.

\begin{corollary}\label{cor:logweighted}
Let \(n\geq 1\) and put \(A_n(v)=\sum_{k=1}^n[k(v-\log k)^2]^{-1}\)
for \(v>\log n\). The inequality \(\mathcal F(x)<0\) holds whenever
\begin{equation}\label{eq:direct-weighted-condition}
  x\geq 17en
  \quad\text{and}\quad
  (1.26)^2 L A_n(L)\leq 1 .
\end{equation}
In particular, \(W_5(x)^2<(ex/L)W_5(x/e)\) for every real
\(x\geq 85e\).
\end{corollary}

\begin{proof}
Rosser and Schoenfeld \cite[Corollary~1]{rosserschoenfeld1962}
give \(\pi(y)<1.25506y/\log y\) for \(y>1\) and
\(\pi(y)>y/\log y\) for \(y\geq 17\). For \(x\geq 17en\)
we may apply the upper bound at \(x/k\) and the lower bound
at \(x/(ek)\) for every \(1\leq k\leq n\). Division by the
logarithm in each summand of \(W_n\) gives
\[
  W_n(x)<1.26xA_n(L),\qquad
  W_n(x/e)>\frac{x}{e}A_n(L-1)>\frac{x}{e}A_n(L).
\]
The last inequality uses \(L-1>\log n\), so all denominators
are positive and each summand decreases with its argument.
It follows from \eqref{eq:direct-weighted-condition} that
\[
  W_n(x)^2<(1.26)^2x^2A_n(L)^2
    \leq \frac{x^2}{L}A_n(L)
    <\frac{ex}{L}W_n(x/e).
\]

For \(n=5\), the condition \(x\geq 85e\) gives \(L\geq
1+\log 85>27/5\). For \(b\geq 0\) the function \(v/(v-b)^2\)
decreases on \(v>b\), since its derivative is
\(-(v+b)/(v-b)^3\). Using \(\log 2<7/10\),
\(\log 5<5/3\), and \(1/3+1/4+1/5=47/60\), we find
\begin{equation*}
  (1.26)^2 L A_5(L)
  \leq (1.26)^2\left[
        \frac{5}{27}
        +\frac{(27/5)/2}{(27/5-7/10)^2}
        +\frac{(47/60)(27/5)}{(27/5-5/3)^2}\right]
   <0.970.
\end{equation*}
This verifies the second condition of \eqref{eq:direct-weighted-condition}
throughout the stated real interval.
\end{proof}

A sufficient condition for \eqref{eq:direct-weighted-condition}
in terms of \(H_n\) and \(\log n\) is
\begin{equation}\label{eq:direct-weighted-explicit}
  L\geq \max\{1+\log(17n),\ (1.26)^2H_n+2\log n\}.
\end{equation}
Indeed, \(A_n(L)\leq H_n/(L-\log n)^2\), and
\((L-\log n)^2/L\geq L-2\log n\). The estimate at \(n=5\)
uses the separate terms for \(k=1,2\) and the smaller sum of
the remaining reciprocal weights.

Theorem~\ref{thm:transfer} also yields a range from any cutoff
\(T\) for the original Ramanujan inequality. Suppose that
\begin{equation}\label{eq:logweighted-range}
  L\geq\max\{\log(nT),\,4H_n\}.
\end{equation}
Write \(L_k=\log(x/k)\) and substitute \(a_k=k\),
\(w_k=L_k^{-1}\) into \eqref{eq:transfer}. Since \(x/k\geq T\),
\[
  W_n(x)^2
  < ex\left(\sum_{k=1}^n\frac{1}{kL_k^2}\right)
       \left(\sum_{k=1}^n\frac{\pi(x/(ek))}{L_k}\right).
\]
Because \(\log n\leq H_n\) and \(L\geq 4H_n\), we have
\(L_k\geq L/2>1\). The first factor in parentheses is at most
\(4H_n/L^2\leq 1/L\). Replacing \(L_k\) by \(L_k-1\) in
the denominator of each summand in the second factor gives
the upper bound \(W_n(x/e)\). Thus \(\mathcal F(x)<0\) also
follows from \eqref{eq:logweighted-range}. With
\(T=e^{490}\) as in \cite[Theorem~6.4]{de2026},
\eqref{eq:logweighted-range} for \(n=5\) reads \(x\geq 5e^{490}\),
because \(4H_5=137/15<490+\log 5\).

\subsection{The Leading Coefficient Of The Weighted Gap}
With fixed weights, the difference between the two sides of
\eqref{eq:transfer} has an asymptotic expansion on the
scale of the individual Ramanujan gaps. To make the dependence
on the shifts explicit, define
\[
  S_w(x)=\sum_{k=1}^n w_k\pi(x/a_k),\qquad
  B_w(x)=\sum_{k=1}^n\frac{w_k}{a_k\log(x/a_k)}
\]
and, for \(x>\max_k a_k\), put
\begin{equation}\label{eq:transferred-gap}
  \Delta_w(x)=exB_w(x)S_w(x/e)-S_w(x)^2.
\end{equation}
The normalization \(w_k/a_k\) arises because the argument
\(x/a_k\) contributes a factor \(a_k^{-1}\) to the leading
prime-counting term.

\begin{theorem}\label{thm:weighted-coefficient}
Let \(a_k>0\) and \(w_k\geq 0\) be fixed, with at least one
\(w_k>0\). Set
\[
  c=\sum_{k=1}^n\frac{w_k}{a_k},\qquad
  p_k=\frac{w_k}{ca_k},\qquad t_k=\log a_k,\qquad
  \bar t=\sum_{k=1}^n p_kt_k.
\]
Then, as \(x\to\infty\),
\begin{equation}\label{eq:weighted-variance}
  \Delta_w(x)
   = c^2\left(1+\sum_{k=1}^n p_k(t_k-\bar t)^2\right)
       \frac{x^2}{L^6}
     +O_{a,w}\!\left(\frac{x^2}{L^7}\right).
\end{equation}
\end{theorem}

\begin{proof}
Write \(y_k=x/a_k\), \(L_k=L-t_k\), and
\(\beta_k=w_k/(a_k L_k)\). At the shift \(a_k\), normalize
the two counts by setting
\[
  U_k=\frac{L_k}{y_k}\pi(y_k),\qquad
  V_k=\frac{eL_k}{y_k}\pi(y_k/e).
\]
Now \(B_w=\sum_k\beta_k\),
\(S_w(x)=x\sum_k\beta_kU_k\), and
\(S_w(x/e)=(x/e)\sum_k\beta_kV_k\). Adding and
subtracting \(B_w\sum_k \beta_kU_k^2\) gives
\begin{equation}\label{eq:weighted-decomposition}
  \frac{\Delta_w(x)}{x^2}
   = B_w\sum_k\beta_k(V_k-U_k^2)
     +\sum_{i<j}\beta_i\beta_j(U_i-U_j)^2.
\end{equation}
The first sum collects the normalized Ramanujan gaps at the
individual shifts. The second is the Cauchy--Schwarz deficit
\((\sum_k\beta_k)\sum_k\beta_kU_k^2-(\sum_k\beta_kU_k)^2\),
written as a sum of squared differences among the normalized counts.

From \eqref{eq:gap-order} at \(y_k\) we have
\[
  V_k-U_k^2
   =-\frac{L_k^2}{y_k^2}\mathcal G(y_k)
   =\frac{1}{L_k^4}+O(L_k^{-5}).
\]
Since \(\beta_k=cp_k/L+O_{a,w}(L^{-2})\) and
\(L_k=L-t_k\), the first term in
\eqref{eq:weighted-decomposition} is \(c^2/L^6+O_{a,w}(L^{-7})\).
For the second term, \eqref{eq:shift-full} gives
\[
  U_k=1+\frac{1}{L_k}+\frac{2}{L_k^2}+O(L_k^{-3})
     =1+\frac{1}{L}+\frac{2+t_k}{L^2}+O_a(L^{-3}).
\]
Hence \(U_i-U_j=(t_i-t_j)/L^2+O_a(L^{-3})\).
The second term in \eqref{eq:weighted-decomposition} has
leading coefficient \(c^2\sum_{i<j}p_ip_j(t_i-t_j)^2\)
at order \(L^{-6}\). Using \(\sum_kp_k=1\), expansion
of the squares gives
\[
  \sum_{i<j}p_ip_j(t_i-t_j)^2
   =\sum_k p_kt_k^2-\left(\sum_k p_kt_k\right)^2
   =\sum_k p_k(t_k-\bar t)^2,
\]
which proves \eqref{eq:weighted-variance}.
\end{proof}

With a single active shift, \(\Delta_w(x)=-w_1^2\mathcal G(x/a_1)\)
after indexing that shift by \(1\). The variance vanishes in
this case. More generally, the variance is zero precisely when all
shifts with positive weight coincide. For a prescribed rescaled
total weight \(c\), the leading coefficient is therefore minimized
when the active shifts coincide. For \(a_k=k\) and \(w_k=1\)
with \(1\leq k\leq 5\),
\(c=H_5=137/60\) and \(p_k=1/(H_5 k)\). These five shifts
give a variance of approximately \(0.352779\), so the leading
coefficient of \(\Delta_w(x)\), on the scale \(x^2/L^6\),
is approximately \(7.052863\).
Section~\ref{sec:variance-samples} compares two pairs of shifts at
the same \(c\) and \(\bar t\), so their leading coefficients differ
by \(c^2\) times the difference of their variances.

Replacing the shifted logarithms in \(B_w\) by \(L\)
changes the multiplier to \(c/L\). The resulting difference is
\begin{equation}\label{eq:unshifted-multiplier}
  \frac{ex c}{L}S_w(x/e)-S_w(x)^2
   =-c^2\bar t\frac{x^2}{L^3}
       +O_{a,w}\!\left(\frac{x^2}{L^4}\right).
\end{equation}
Indeed, \(B_w=c/L+c\bar t/L^2+O_{a,w}(L^{-3})\) and
\(S_w(x/e)=cx/(eL)+O_{a,w}(x/L^2)\). Subtracting the
term \(ex(B_w-c/L)S_w(x/e)\) from \(\Delta_w\) gives
\eqref{eq:unshifted-multiplier}, since \(\Delta_w\)
has order \(x^2/L^6\). In particular, a positive mean
logarithmic shift gives a negative gap with the multiplier \(c/L\).

\section{A Coefficient Rule For Polynomial Comparisons}\label{sec:framework}
\subsection{Finite Combinations Of Shifted Products}
The computations in Sections~\ref{sec:polynomials} and \ref{sec:sums}
begin by comparing powers of \(x\). When two powers agree, the
logarithmic exponents determine which terms contribute to the
leading coefficient. The next proposition gives both the leading
coefficient at that scale and its first correction.

\begin{proposition}\label{prop:coefficients}
Consider a finite sum
\begin{equation}\label{eq:shifted-products}
  D(x)=\sum_{j=1}^{J} c_j\frac{x^{u_j}}{L^{v_j}}
  \prod_{i=1}^{s_j}\pi(x/a_{ji})^{r_{ji}},
\end{equation}
with fixed \(c_j,u_j\in\mathbb R\), \(v_j\in\mathbb Z\),
\(a_{ji}>0\), and nonnegative integers \(r_{ji}\). An empty
product has value \(1\). Put
\[
\begin{gathered}
  R_j=\sum_i r_{ji},\qquad
  A_j=u_j+R_j,\qquad B_j=v_j+R_j,\\
  c_j^*=c_j\prod_i a_{ji}^{-r_{ji}},\qquad
  d_j=\sum_i r_{ji}(1+\log a_{ji}).
\end{gathered}
\]
Let \(A=\max_j A_j\) and \(B=\min_{j:A_j=A} B_j\).
Then
\begin{equation}\label{eq:coefficient-expansion}
  D(x)=\frac{x^A}{L^B}\left(C_0+\frac{C_1}{L}+O(L^{-2})\right),
\end{equation}
where
\begin{equation}\label{eq:coefficients}
  C_0=\!\!\sum_{\substack{A_j=A\\B_j=B}} c_j^*,\qquad
  C_1=\!\!\sum_{\substack{A_j=A\\B_j=B}} c_j^*d_j
      +\!\!\sum_{\substack{A_j=A\\B_j=B+1}} c_j^* .
\end{equation}
If \(C_0\ne 0\), it determines the eventual sign. If \(C_0=0\)
and \(C_1\ne 0\), the leading term is \(C_1x^A/L^{B+1}\).
\end{proposition}

\begin{proof}
In the \(j\)th summand, the leading term of every prime count
contributes a factor \(x/(a_{ji}L)\). Raising
\eqref{eq:shift-two} to the power \(r_{ji}\) multiplies its
coefficient of \(L^{-1}\) by \(r_{ji}\). On multiplying the
resulting factors, the coefficient of \(L^{-1}\) is their sum.
Thus the \(j\)th summand equals
\[
  c_j^*\frac{x^{A_j}}{L^{B_j}}
       \left(1+\frac{d_j}{L}+O(L^{-2})\right).
\]
If \(A_j<A\), its contribution after division by \(x^A/L^B\)
is smaller than \(L^{-2}\), since any fixed negative power of
\(x\) dominates any fixed power of \(L\) in the decay rate.
Among terms with \(A_j=A\), those with \(B_j=B\) contribute to
both coefficients, and those with \(B_j=B+1\) contribute only
to \(C_1\). The remaining terms have \(B_j\geq B+2\) and belong
to the remainder. This proves \eqref{eq:coefficient-expansion}.
\end{proof}

\subsection{Cancellation In A Two-Polynomial Comparison}
Let \(P\) and \(Q\) be nonzero polynomials with degrees \(p,q\)
and leading coefficients \(a,b\), respectively. Consider
\begin{equation}\label{eq:general}
  P(\pi(x))-\frac{ex}{L}Q(\pi(x/e))+R(x).
\end{equation}
The first term has leading scale \((x/L)^p\) and coefficient \(a\).
The second has scale \((x/L)^{q+1}\) and coefficient \(-b e^{1-q}\).
If \(R\) is of smaller order than the larger of these scales,
cancellation at that scale requires \(p=q+1\) and \(a=b e^{1-q}\).
If \(R\) also has a leading term at that scale, its coefficient
must be included when testing for cancellation.

In the case \(d=p=q+1\), keeping the next term gives
\begin{equation}\label{eq:balanced}
  P(\pi(x))-\frac{ex}{L}Q(\pi(x/e))
  =\frac{x^d}{L^d}\left[
     a-b e^{1-q}+\frac{da-2q\,b e^{1-q}}{L}
     +O(L^{-2})\right].
\end{equation}
Here \(\pi(x)\) and \(e\pi(x/e)\) have respective first
corrections \(1/L\) and \(2/L\) relative to \(x/L\), by
\eqref{eq:shift-two}. Each lower power in either polynomial
is smaller than \(x^d/L^{d+2}\) after the prefactors in
\eqref{eq:balanced} are applied. This accounts for every term
in the expansion.

For \(a=b e^{1-q}\), the second coefficient reduces to \(a(2-d)\).
In particular, if \(d\geq 3\) and \(R=o(x^d/L^{d+1})\), the
leading term in \eqref{eq:general} is
\(a(2-d)x^d/L^{d+1}\). For example, set \(P(z)=z^d\),
\(Q(z)=e^{d-2}z^{d-1}\), and \(R=0\). For every fixed integer
\(d\geq 3\), this gives
\[
  \pi(x)^d-\frac{e^{d-1}x}{L}\pi(x/e)^{d-1}
    =-(d-2)\frac{x^d}{L^{d+1}}
       +O_d\!\left(\frac{x^d}{L^{d+2}}\right).
\]
At \(d=3\), the eventual sign recovers Hassani's upper bound
in \eqref{eq:hassani-cubic}. When \(d=2\), both coefficients of
the difference in \eqref{eq:balanced} vanish. The example
\(P(z)=z^2\), \(Q(z)=z\), \(R=0\) then gives the smaller
Ramanujan gap \eqref{eq:gap-order}.

\subsection{Moment Cancellation Among Equal-Degree Terms}\label{sec:moments}
In a sum of shifted powers of degree \(d\), a factor \(a^d\)
compensates for the leading change in \(\pi(x/a)^d\) under
dilation. Once the terms have this common leading scale, their
coefficients can be chosen to cancel several successive orders.
The logarithms of the shifts determine these choices.

\begin{theorem}\label{thm:moments}
Let \(d\geq 1\) be an integer, let \(t_0,\ldots,t_s\) be distinct
real numbers with \(s\geq 0\), and let \(b_0,\ldots,b_s\) be
real coefficients with at least one nonzero. All parameters are
fixed as \(x\to\infty\). Define
\[
  D_d(x)=\sum_{i=0}^s b_i e^{dt_i}\pi(x/e^{t_i})^d,
  \qquad M_j=\sum_{i=0}^s b_i t_i^j
  \quad(j\geq 0),
\]
where \(t_i^0=1\). The first nonzero moment \(M_r\) has \(r\leq s\),
and
\begin{equation}\label{eq:moment-asymptotic}
  D_d(x)
    =\binom{d+r-1}{r} M_r \frac{x^d}{L^{d+r}}
       +O\!\left(\frac{x^d}{L^{d+r+1}}\right).
\end{equation}
Thus the first \(r\) coefficients on the scale \(x^d/L^d\)
vanish, and the sign of \(M_r\) determines the eventual sign.
\end{theorem}

\begin{proof}
Suppose \(M_j=0\) for \(0\leq j\leq s\). Then
\(\sum_i b_i f(t_i)=0\) for every polynomial \(f\) of degree
at most \(s\). The Lagrange basis polynomial
\[
  \ell_i(t)=\prod_{\substack{0\leq j\leq s\\ j\ne i}}
      \frac{t-t_j}{t_i-t_j}
\]
has this degree bound and takes value \(1\) at \(t_i\) and
value \(0\) at the other nodes. Substituting \(f=\ell_i\)
gives \(b_i=0\) for every \(i\), contradicting the assumption.
It follows that \(r\leq s\).

Fix a nonnegative integer \(J\) and set
\(U(t,x)=e^t L\pi(x/e^t)/x\). By \eqref{eq:shift-full},
\[
  U(t,x)=\sum_{j=0}^J j! L^{-j}(1-t/L)^{-j-1}+O(L^{-J-1}).
\]
Expanding each denominator by the binomial series gives
the coefficient of \(L^{-m}\) for \(0\leq m\leq J\) as
\[
  F_m(t)=\sum_{j=0}^m j!\binom{m}{j}t^{m-j}.
\]
In particular, \(F_m\) is monic of degree \(m\). Denote by
\(Q_{d,m}(t)\) the coefficient of \(L^{-m}\) in \(U(t,x)^d\).
Each product \(F_{m_1}(t)\cdots F_{m_d}(t)\) with
\(m_1+\cdots+m_d=m\) is monic of degree \(m\). The number of
such \(d\)-tuples of nonnegative integers is \(\binom{d+m-1}{m}\),
which is therefore the leading coefficient of \(Q_{d,m}\).
The remainders are uniform across the finite set of nodes.

Dividing \(D_d(x)\) by \(x^d/L^d\) leaves
\(\sum_i b_i U(t_i,x)^d\). Taking \(J=r\), the coefficient
at \(L^{-m}\) for \(0\leq m\leq r\) is
\(\sum_i b_i Q_{d,m}(t_i)\). This vanishes
for \(m<r\) because every monomial in \(Q_{d,m}\) has degree
at most \(m\). For \(m=r\) the lower-degree monomials cancel,
leaving \(\binom{d+r-1}{r}M_r\). This gives
\eqref{eq:moment-asymptotic}.
\end{proof}

The maximum of \(s\) cancellations can be attained at any set
of \(s+1\) distinct nodes. For \(s\geq 1\), the choice
\begin{equation}\label{eq:divided-difference-weights}
  b_i=\frac{1}{\displaystyle\prod_{j\ne i}(t_i-t_j)}
\end{equation}
gives \(M_0=\cdots=M_{s-1}=0\) and \(M_s=1\).
Indeed, \(\sum_i f(t_i)\ell_i(t)=f(t)\) for \(\deg f\leq s\).
The coefficient of \(t^s\) on the left is \(\sum_i b_i f(t_i)\).
Taking \(f(t)=t^j\) gives the asserted moments.

For example, take \(d=3\), \(s=4\), and \(t_i=ih\) with fixed
\(h>0\). Multiplying the weights in
\eqref{eq:divided-difference-weights} by \(4!h^4\) gives
\((-1)^i\binom{4}{i}\). Their fourth moment is \(24h^4\),
so \eqref{eq:moment-asymptotic} yields
\begin{equation}\label{eq:cubic-four-cancellations}
  \mathcal C_h(x)=\sum_{i=0}^4(-1)^i\binom{4}{i}e^{3ih}\pi(x/e^{ih})^3
  =360h^4\frac{x^3}{L^7}+O_h\!\left(\frac{x^3}{L^8}\right).
\end{equation}
The terms of orders \(x^3/L^3\) through \(x^3/L^6\) cancel.
In both this cubic comparison and the original quadratic one,
the first nonzero term lies four powers of \(L\) below the
common leading scale.

\subsection{Higher-Degree Examples}
For each fixed integer \(m\geq 3\), define
\begin{equation}\label{eq:general-cubic}
  \mathcal P_m(x)=\pi(x)^m-\frac{3ex}{L}\pi(x/e)^{m-1}
               +\frac{3e^2x}{L^2}\pi(x/e^2)^{m-2}.
\end{equation}
The powers-of-three subfamily is denoted by
\(\mathcal H_n=\mathcal P_{3^n}\) for \(n\geq1\).
In particular, \(\mathcal H_1=\mathcal P_3=\mathcal H\).
For fixed \(n>1\) and a fixed integer \(r\geq 2\), the
corresponding example formed from the sums is
\begin{equation}\label{eq:general-sum}
  \mathcal N_r(x)=S_n(x)^r-\frac{ex}{L}S_n(x/e)^r+S_n(x/e^2)^r.
\end{equation}

\begin{theorem}\label{thm:general-signs}
For each fixed integer \(m\geq 3\), as \(x\to\infty\),
\begin{equation}\label{eq:general-cubic-asymptotic}
  \mathcal P_m(x)
   =\frac{x^m}{L^m}\left[1-3e^{2-m}
       +\frac{m-6(m-1)e^{2-m}}{L}+O_m(L^{-2})\right].
\end{equation}
For fixed integers \(n>1\) and \(r\geq 2\), we have
\begin{equation}\label{eq:general-sum-asymptotic}
  \mathcal N_r(x)
   =-e^{1-r}\frac{x^{r+1}}{L^{r+1}}
     \left[H_n^r+\frac{rH_n^{r-1}(2H_n+J_n)}{L}
                     +O_{n,r}(L^{-2})\right].
\end{equation}
Consequently, \(\mathcal P_3(x)<0\) eventually, and
\(\mathcal P_m(x)>0\) eventually for every fixed \(m\geq4\).
The eventual sign of \(\mathcal N_r(x)\) is negative for
the stated \(n\) and \(r\).
\end{theorem}

\begin{proof}
In \eqref{eq:general-cubic}, the terms at the scale
\(x^m/L^m\) have coefficients \(1\) and \(-3e^{2-m}\).
Their relative corrections are \(m/L\) and \(2(m-1)/L\)
by \eqref{eq:shift-two}, as in
Proposition~\ref{prop:coefficients}. The third summand
has order \(x^{m-1}/L^m\), so its normalized contribution
is absorbed by \(O_m(L^{-2})\). The combined leading
coefficient at \(m=3\) is \(1-3/e<0\). For \(m\geq 4\),
the inequality \(3e^{2-m}\leq 3/e^2<1\) makes the coefficient
positive.

In \eqref{eq:general-sum}, \eqref{eq:shifted-S} shows that
raising \(S_n(x/e)\) to the power \(r\) introduces the
first correction \(rH_n^{r-1}(2H_n+J_n)/L\) in the
bracket after normalization. The factor \(ex/L\) then
puts the middle term on the scale \(x^{r+1}/L^{r+1}\).
Each of the other two terms has order \(x^r/L^r\) and
its normalized contribution is \(O_{n,r}(L/x)\).
Both are absorbed by the remainder in
\eqref{eq:general-sum-asymptotic}. The negative sign
comes from the shifted term for every fixed \(r\).
\end{proof}

\subsection{An Effective Bound For The Balanced Cubic}\label{sec:effective-cubic}
The asymptotic formula \eqref{eq:cubic-four-cancellations}
determines the eventual sign for every fixed \(h>0\).
At \(h=1\) and \(h=2\), the fourth derivative of a smooth
approximation gives a lower bound on the first surviving scale.
Explicit prime-counting estimates then control the error across
the five shifted arguments.

\begin{corollary}\label{cor:effective-cubic}
For the cubic \(\mathcal C_h\) in \eqref{eq:cubic-four-cancellations},
\begin{align}
  \mathcal C_1(x)&>40\,\frac{x^3}{(\log x)^7}
    &&\text{for all real \(x\geq e^{2704}\)},\label{eq:effective-cubic-1}\\
  \mathcal C_2(x)&>750\,\frac{x^3}{(\log x)^7}
    &&\text{for all real \(x\geq e^{52}\)}.\label{eq:effective-cubic-2}
\end{align}
\end{corollary}

\begin{proof}
For \(L\geq52\) and \(0\leq t\leq8\), set \(v=L-t\) and
\[
  P_L(t)=L\int_0^{L-9}\frac{e^{-u}}{L-t-u}\,du .
\]
The upper endpoint is independent of \(t\) and the denominator
is at least \(1\). Since \(L-t-u\leq L\), differentiation under
the integral for \(0\leq j\leq4\) gives
\[
  P_L^{(j)}(t)\geq \frac{j!}{L^j}\bigl(1-e^{-(L-9)}\bigr).
\]
The fourth derivative of \(P_L^3\) is a sum over triples of
nonnegative integers with sum \(4\). Each of the \(15\) terms,
including its multinomial coefficient, is at least
\(4!(1-e^{-(L-9)})^3/L^4\). For a smooth function \(f\) on
\([0,4h]\), repeated use of the fundamental theorem of calculus
gives
\[
  \sum_{i=0}^4(-1)^i\binom{4}{i}f(ih)
    =\int_{[0,h]^4} f^{(4)}(u_1+\cdots+u_4)\,du_1\cdots du_4 .
\]
Thus the fourth difference of \(P_L^3\) is at least
\(360h^4(1-e^{-(L-9)})^3/L^4\) for \(h=1,2\).

To compare \(P_L\) with \(U(t,x)=e^tL\pi(x/e^t)/x\), let
\(\Li_2(y)=\int_2^y ds/\log s=\Li(y)-\Li(2)\) for \(y\geq2\).
This is the normalization of the logarithmic integral in
\cite{fiorikadiriswidinsky2023}. Changing variables to
\(u=\log(e^v/s)\) gives
\[
  Le^{-v}\Li_2(e^v)=L\int_0^{v-\log 2}\frac{e^{-u}}{v-u}\,du .
\]
As \(v-\log 2>L-9\), the omitted part of this integral is at most
\(Le^{-(L-9)}/\log 2<2Le^{-(L-9)}\). Put
\(\eta(v)=ve^{-v}|\pi(e^v)-\Li_2(e^v)|\). It follows that
\begin{equation}\label{eq:effective-cubic-error}
  |U(t,x)-P_L(t)|
    \leq \frac{L}{v}\eta(v)+2Le^{-(L-9)} .
\end{equation}
After multiplication by \(L^4\), the last term is less than
\(10^{-3}\). Indeed, \(L^5e^{-L}\) decreases for \(L\geq52\),
and \(2\cdot52^5e^{-43}<2\cdot52^5/2^{43}<10^{-3}\).

Table~4 of \cite{fiorikadiriswidinsky2023} bounds \(\eta(v)\)
uniformly for \(v\geq v_0\) at the values of \(v_0\) in
Table~\ref{tab:cubic-error}.
The bounds \(\eta_0\) are rounded upward. On each interval
\([a,b)\) for \(L\), we have \(v\geq a-8\) and
\(L/v\leq a/(a-8)\). The last column bounds
\(L^5\eta(v)/v\) by rounding \(b^4a\eta_0/(a-8)\) upward.
\begin{table}[!htbp]
  \caption{Uniform bounds for the normalized prime-counting error
  in the balanced cubic.}
  \label{tab:cubic-error}
  \renewcommand{\arraystretch}{1.12}
  \begin{tabular*}{0.7\textwidth}{@{\extracolsep{\fill}}cccc@{}}
    \toprule
    Interval for \(L\) & \(v_0\) & \(\eta_0\)
      & Upper bound for \(L^5\eta(v)/v\)\\
    \midrule
    \([52,108)\)    & \(44\)   & \(1.79\cdot10^{-8}\)   & \(3\)\\
    \([108,2008)\)  & \(100\)  & \(2.034\cdot10^{-12}\) & \(36\)\\
    \([2008,2308)\) & \(2000\) & \(1.571\cdot10^{-12}\) & \(45\)\\
    \([2308,2708)\) & \(2300\) & \(3.808\cdot10^{-13}\) & \(21\)\\
    \botrule
  \end{tabular*}
\end{table}
\FloatBarrier

For \(v\geq2700\), Corollary~22 of
\cite{fiorikadiriswidinsky2023} gives
\[
  v^4\eta(v)\leq g(v),\qquad
  g(v)=9.2211\,v^{11/2}e^{-0.8476\sqrt v}.
\]
The logarithmic derivative of \(g\) is
\((11-0.8476\sqrt v)/(2v)<0\) for \(v\geq2700\).
Since \(51.96<\sqrt{2700}<52\) and \(0.8476\cdot51.96>44\),
the bound \(e>2.71828\) gives
\[
  g(2700)
    < \frac{9.2211\cdot52\cdot2700^5}{(2.71828)^{44}} < 5.5 .
\]
Whenever \(v\geq2700\) and \(0\leq t\leq8\),
\(L/v\leq2708/2700\), so
\(L^5\eta(v)/v\leq (L/v)^5g(v)<5.6\).
Combining the interval bounds, the bound for \(v\geq2700\)
and \eqref{eq:effective-cubic-error} yields
\begin{equation}\label{eq:effective-cubic-two-errors}
\begin{aligned}
  |U(t,x)-P_L(t)|&<46/L^4 &&(L\geq52,\ 0\leq t\leq8),\\
  |U(t,x)-P_L(t)|&<6/L^4 &&(L\geq2704,\ 0\leq t\leq4).
\end{aligned}
\end{equation}
In the second line \(v=L-t\geq2700\) across the four unit steps.

Splitting the integral for \(P_L(t)\) at \(L/4\) gives
\[
  0<P_L(t)\leq 1+\frac{9}{3L/4-8}+Le^{-L/4}.
\]
On the first part use \(L/(L-t-u)=1+(t+u)/(L-t-u)\)
and \(\int_0^\infty(8+u)e^{-u}\,du=9\). The denominator
on the second part is at least \(1\). The displayed upper
bound decreases with \(L\geq52\), and \(Le^{-L/4}<52/2^{13}\)
there. Combining the two lines of
\eqref{eq:effective-cubic-two-errors} with this estimate
bounds both \(P_L(t)\) and \(U(t,x)\) by
\(1+9/31+52/2^{13}+46/52^4<3/2\) on the first range and by
\(1+9/2020+52/2^{13}+6/2704^4<21/20\) on the second.

For nonnegative \(p,q\leq M\), factorization yields
\(|p^3-q^3|\leq3M^2|p-q|\). The five weights have
absolute sum \(16\), so the error between the fourth
differences of \(U^3\) and \(P_L^3\) is at most
\(48M^2\sup_{0\leq t\leq4h}|U(t,x)-P_L(t)|\).
The fourth difference of \(U^3\) is \(L^3\mathcal C_h(x)/x^3\).
Using \(e^{-(L-9)}<10^{-3}\) in the lower bound for the
smooth fourth difference, we obtain
\[
\begin{aligned}
  \frac{L^7}{x^3}\mathcal C_1(x)
     &>360(0.999)^3-48(21/20)^2\cdot6>40,\\
  \frac{L^7}{x^3}\mathcal C_2(x)
     &>5760(0.999)^3-48(3/2)^2\cdot46>750
\end{aligned}
\]
on the respective real ranges in the statement.
\end{proof}

\section{Numerical Illustrations}\label{sec:numerics}
Tables~\ref{tab:basic} and \ref{tab:general} report the
expressions divided by their leading scales. We take \(n=5\)
in \(S_n,W_n,\mathcal L,\mathcal F\), and \(\mathcal N_r\), so
\(H_5=137/60\). We evaluate \(\mathcal H_2=\mathcal P_9\)
and \(\mathcal H_3=\mathcal P_{27}\) for the higher-degree
family. In Figures~\ref{fig:basic} and \ref{fig:general}, the dashed
lines are the coefficients in the asymptotic formulas, including
the harmonic-number factors for the sums. Table~\ref{tab:variance}
uses the same large arguments to compare two normalized weighted
Ramanujan gaps.

\subsection{Prime Counts And Interval Enclosures}
At \(x=10^j\) for \(4\leq j\leq 13\), we compute each required
prime count at the integer part of its argument. Interval
arithmetic determines the integer parts of the nonintegral
arguments. The accompanying code counts primes at those integers
using Lehmer's exact counting formula \cite{lehmer1959}.
It also propagates interval bounds for the logarithms and
exponentials through the polynomial expressions.

For the larger samples \(x=\exp(L)\), put \(v=L-\log a\)
at a shift \(a\) and write
\[
  B_a(x)=\frac{L}{x}\pi(x/a),\qquad
  P_5(u)=\sum_{j=0}^5 j!\,u^j .
\]
The explicit estimate \eqref{eq:explicit-pi} and the bound
\(\bigl|v e^{-v}\Li(e^v)-P_5(1/v)\bigr|<900v^{-6}\)
from \cite[Eq.~(40)]{de2026} give
\begin{equation}\label{eq:numerical-enclosure}
  \left|B_a(x)-\frac{L}{a v}P_5(1/v)\right|
  <\frac{L}{a}\left(\frac{900}{v^7}
             +0.1e^{-0.07v}+21\,000e^{-0.087v}\right).
\end{equation}
The largest shift used in the tables is \(5e^2\). For the selected
values \(L\geq 500\), this gives \(v\geq 500-2-\log 5>489\),
as required by \eqref{eq:explicit-pi}. Substituting the enclosures
for \(B_a\) into the expressions in Tables~\ref{tab:basic} and
\ref{tab:general} determines their entries uniquely to six
decimal places. The figures show the centers of these intervals.

\subsection{Interpretation Of The Samples}
The scale factors are positive, so the tabulated signs agree with
the original signs at each sampled argument. The first corrections
help explain the distances from the limiting coefficients in the
figures. For example, \eqref{eq:general-cubic-asymptotic} gives
\[
  \frac{L^{27}}{x^{27}}\mathcal H_3(x)
    =1-3e^{-25}+\frac{27-156e^{-25}}{L}+O(L^{-2}).
\]
At \(L=3300\), the first correction is about \(0.008182\),
compared with a deviation of about \(0.008219\) from the limit
in Table~\ref{tab:general}. This accounts for the gradual approach
of \(\mathcal H_3\) on the normalized scale. The effective sign
of \(\mathcal F\) is already settled at every tabulated argument
by Corollary~\ref{cor:logweighted}, since the first argument in
Table~\ref{tab:basic} is \(10^4>85e\).

The sampled values describe the expressions at discrete arguments.
For instance, at any prime \(p\) the shifted counts in
\eqref{eq:cubic} are continuous because \(e\) and \(e^2\)
are irrational. The unshifted count jumps by one, so
\[
  \mathcal H(p)-\lim_{x\uparrow p}\mathcal H(x)
    =\pi(p)^3-(\pi(p)-1)^3>0.
\]
Thus the behavior between sampled arguments includes upward jumps,
even when successive tabulated values of the unnormalized cubic
decrease.

\begin{table}[p]
  \caption{Normalized cubic, quartic, and sum expressions.
  The sum expressions use \(n=5\).}
  \label{tab:basic}
  \renewcommand{\arraystretch}{1.12}
\begin{tabular*}{\textwidth}{@{\extracolsep{\fill}}ccccc@{}}
\toprule
$x$ & $\mathcal{H}(x)L^{3}/x^{3}$ & $\mathcal{K}(x)L^{4}/x^{4}$ & $\mathcal{L}(x)L^{2}/x^{2}$ & $\mathcal{F}(x)L^{3}/x^{2}$ \\
\midrule
$10^{4}$ & $-0.376825$ & $0.488298$ & $4.617214$ & $-2.947713$ \\
$10^{5}$ & $-0.298115$ & $0.502242$ & $4.218909$ & $-2.792870$ \\
$10^{6}$ & $-0.256909$ & $0.497550$ & $3.954959$ & $-2.690723$ \\
$10^{7}$ & $-0.225001$ & $0.497881$ & $3.775243$ & $-2.619911$ \\
$10^{8}$ & $-0.204873$ & $0.495018$ & $3.647238$ & $-2.569545$ \\
$10^{9}$ & $-0.190676$ & $0.491783$ & $3.552146$ & $-2.532066$ \\
$10^{10}$ & $-0.179925$ & $0.488955$ & $3.479144$ & $-2.503247$ \\
$10^{11}$ & $-0.171520$ & $0.486513$ & $3.421492$ & $-2.480400$ \\
$10^{12}$ & $-0.164776$ & $0.484413$ & $3.374792$ & $-2.461870$ \\
$10^{13}$ & $-0.159252$ & $0.482589$ & $3.336207$ & $-2.446539$ \\
\midrule
$\exp(500)$ & $-0.106493$ & $0.460160$ & $2.951996$ & $-2.292182$ \\
$\exp(700)$ & $-0.105672$ & $0.459732$ & $2.945761$ & $-2.289643$ \\
$\exp(1000)$ & $-0.105059$ & $0.459410$ & $2.941101$ & $-2.287745$ \\
$\exp(1500)$ & $-0.104584$ & $0.459160$ & $2.937485$ & $-2.286272$ \\
$\exp(2300)$ & $-0.104255$ & $0.458986$ & $2.934975$ & $-2.285248$ \\
$\exp(3300)$ & $-0.104068$ & $0.458887$ & $2.933550$ & $-2.284667$ \\
\midrule
Limit & $-0.103638$ & $0.458659$ & $2.930278$ & $-2.283333$ \\
\botrule
\end{tabular*}
\end{table}

\begin{figure}[p]
  \centering
  \includegraphics[width=\linewidth]{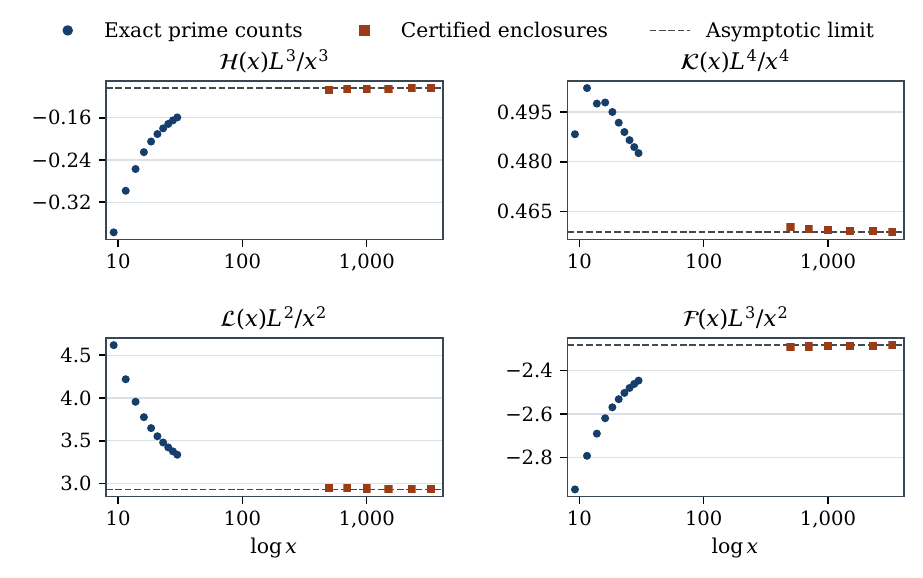}
  \caption{The four normalized expressions in Table~\ref{tab:basic}
  at the sampled arguments. The horizontal axes use logarithmic
  spacing in \(\log x\).}
  \label{fig:basic}
\end{figure}

\FloatBarrier

\begin{table}[p]
  \caption{Normalized higher-degree examples. The values for
  \(\mathcal N_r\) use \(n=5\).}
  \label{tab:general}
  \renewcommand{\arraystretch}{1.12}
\begin{tabular*}{\textwidth}{@{\extracolsep{\fill}}ccccc@{}}
\toprule
$x$ & $\mathcal{H}_2(x)L^{9}/x^{9}$ & $\mathcal{H}_3(x)L^{27}/x^{27}$ & $\mathcal{N}_3(x)L^{4}/x^{4}$ & $\mathcal{N}_4(x)L^{5}/x^{5}$ \\
\midrule
$10^{4}$ & $3.030467$ & $28.401691$ & $-4.465105$ & $-5.243822$ \\
$10^{5}$ & $2.429109$ & $14.573414$ & $-3.608440$ & $-3.963068$ \\
$10^{6}$ & $2.064888$ & $8.934893$ & $-3.111295$ & $-3.254247$ \\
$10^{7}$ & $1.848462$ & $6.400726$ & $-2.798853$ & $-2.826206$ \\
$10^{8}$ & $1.701118$ & $4.984504$ & $-2.590913$ & $-2.549794$ \\
$10^{9}$ & $1.595275$ & $4.108266$ & $-2.443539$ & $-2.358273$ \\
$10^{10}$ & $1.516506$ & $3.527642$ & $-2.334289$ & $-2.218748$ \\
$10^{11}$ & $1.455814$ & $3.119704$ & $-2.250168$ & $-2.112784$ \\
$10^{12}$ & $1.407758$ & $2.820046$ & $-2.183513$ & $-2.029752$ \\
$10^{13}$ & $1.368778$ & $2.591610$ & $-2.129415$ & $-1.962978$ \\
\midrule
$\exp(500)$ & $1.015393$ & $1.055657$ & $-1.636559$ & $-1.381903$ \\
$\exp(700)$ & $1.010169$ & $1.039411$ & $-1.629221$ & $-1.373648$ \\
$\exp(1000)$ & $1.006274$ & $1.027410$ & $-1.623750$ & $-1.367500$ \\
$\exp(1500)$ & $1.003259$ & $1.018181$ & $-1.619513$ & $-1.362744$ \\
$\exp(2300)$ & $1.001169$ & $1.011816$ & $-1.616575$ & $-1.359449$ \\
$\exp(3300)$ & $0.999983$ & $1.008219$ & $-1.614909$ & $-1.357581$ \\
\midrule
Limit & $0.997264$ & $1.000000$ & $-1.611087$ & $-1.353299$ \\
\botrule
\end{tabular*}
\end{table}

\begin{figure}[p]
  \centering
  \includegraphics[width=\linewidth]{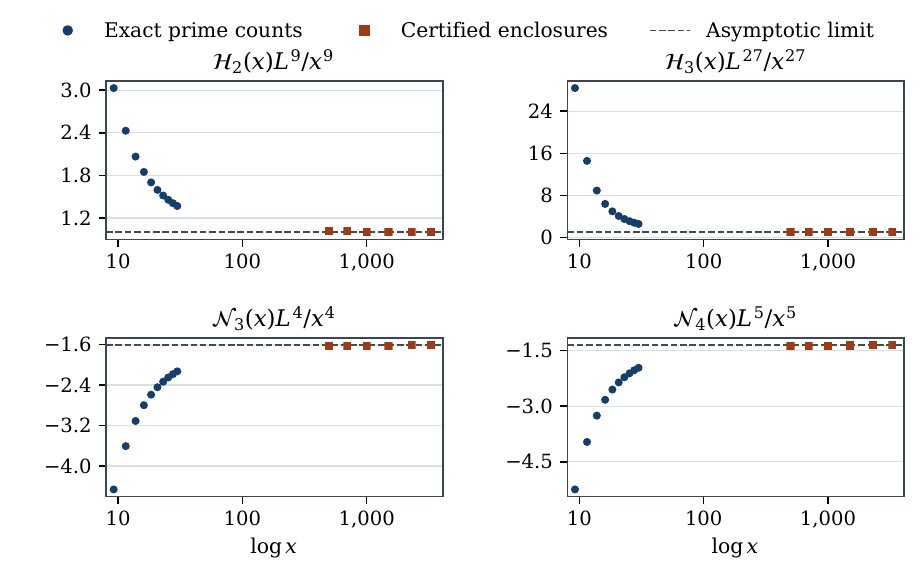}
  \caption{The four normalized expressions in Table~\ref{tab:general}
  at the sampled arguments. The horizontal axes use logarithmic
  spacing in \(\log x\).}
  \label{fig:general}
\end{figure}

\FloatBarrier

\subsection{Comparison At Fixed Mean And Total Weight}\label{sec:variance-samples}
To isolate the variance in Theorem~\ref{thm:weighted-coefficient},
take the pair \(t_-=1-h\), \(t_+=1+h\) with
\(h\in\{1/2,1\}\). Set \(a_\pm=e^{t_\pm}\) and
\(w_\pm=a_\pm/2\). In both cases \(c=1\), \(p_\pm=1/2\), and
\(\bar t=1\), while the weighted variance is \(h^2\). Write
\(\Delta_h\) for the gap \eqref{eq:transferred-gap} with these
weights and shifts, and define
\begin{equation}\label{eq:variance-normalized}
  R_h(L)=L^6e^{-2L}\Delta_h(e^L).
\end{equation}
By \eqref{eq:weighted-variance}, \(R_h(L)=1+h^2+O_h(L^{-1})\).
The limiting values at \(h=1/2\) and \(h=1\) are \(5/4\) and
\(2\), respectively.

For the interval calculation, let \(v_\pm=L-t_\pm\) and
\(B(v)=ve^{-v}\pi(e^v)\). The bounds used in
\eqref{eq:numerical-enclosure} enclose \(B(v)\) in an interval
centered at \(P_5(1/v)\) with radius
\[
  \frac{900}{v^6}
  +v\bigl(0.1e^{-0.07v}+21\,000e^{-0.087v}\bigr).
\]
The smallest logarithmic argument used here is \(L-3\geq 497\),
so the bound applies at \(v_\pm\) and \(v_\pm-1\). In
\eqref{eq:weighted-decomposition} use
\[
  \beta_\pm=\frac{1}{2v_\pm},\qquad
  U_\pm=B(v_\pm),\qquad
  V_\pm=\frac{v_\pm}{v_\pm-1}B(v_\pm-1).
\]
Multiplying the resulting interval for \(\Delta_h(x)/x^2\)
by \(L^6\) gives an enclosure for \(R_h(L)\). The endpoints
in Table~\ref{tab:variance} are rounded outward to six decimal
places. At each sampled argument, the intervals are disjoint and establish
the larger gap for \(h=1\). At \(L=3300\), their difference
\(R_1(L)-R_{1/2}(L)\) lies between \(0.752694\) and \(0.753691\),
compared with its limit \(3/4\).

\begin{table}[!htbp]
  \centering
  \caption{Enclosures for \(R_h(L)\) at rescaled total weight \(c=1\)
  and mean logarithmic shift \(\bar t=1\). The variance is \(h^2\).}
  \label{tab:variance}
  \renewcommand{\arraystretch}{1.12}
\begin{tabular*}{0.7\textwidth}{@{\extracolsep{\fill}}ccc@{}}
\toprule
$L$ & $R_{1/2}(L)$ & $R_1(L)$ \\
\midrule
500 & $[0.998523,\,1.603812]$ & $[1.769156,\,2.376092]$ \\
700 & $[1.280087,\,1.291283]$ & $[2.045318,\,2.056515]$ \\
1000 & $[1.271801,\,1.277262]$ & $[2.032413,\,2.037874]$ \\
1500 & $[1.264914,\,1.267333]$ & $[2.021964,\,2.024383]$ \\
2300 & $[1.259898,\,1.260925]$ & $[2.014484,\,2.015511]$ \\
3300 & $[1.256966,\,1.257464]$ & $[2.010158,\,2.010657]$ \\
\midrule
Limit & $5/4$ & $2$ \\
\botrule
\end{tabular*}
\end{table}

\FloatBarrier
\section{Discussion}\label{sec:discussion}
The order of cancellation determines the precision needed in
an effective estimate. For \(\mathcal F\), the different
asymptotic orders allow the use of the Rosser--Schoenfeld bounds
in Corollary~\ref{cor:logweighted}. For the balanced cubic,
the fourth derivative of the smooth model gives a lower bound
on the scale \(L^{-4}\) after normalization by \(x^3/L^3\).
Corollary~\ref{cor:effective-cubic} controls the error on this
scale over full real intervals. For a degree-\(d\) comparison,
the error at each argument is bounded by
\(d\,|U-P_L|\max\{|U|,|P_L|\}^{d-1}\) when \(P_L\)
approximates the normalized prime count \(U\).
The first nonvanishing moment sets the leading coefficient of
the comparison, while the absolute values of its weights govern
how the errors at the individual arguments accumulate.

The transfer inequality \eqref{eq:transfer} depends on a
pointwise range for Ramanujan's inequality and on the largest
dilation. It applies even when the number of arguments or the
weights change with \(x\), provided each shifted argument meets
the range condition. Extending the fixed-parameter asymptotic
formulas to growing families of shifts would require estimates
uniform in those shifts.

\section*{Acknowledgments}
I thank Adrian W. Dudek for encouraging my interest in
Ramanujan's inequality and for the perspective I gained from
his work on the problem.

\section*{Data And Code Availability}
I as the sole author of this article confirm that the data supporting the findings of this study are available within the article [and/or] its supplementary materials.

\section*{Conflicts Of Interest}
The author declares no conflicts of interest.

\end{document}